\documentclass[dvipdfmx]{amsart}
\usepackage{amssymb, amsmath, amscd,latexsym, amsthm}
\usepackage{graphicx}
\usepackage[abs]{overpic} %
\usepackage{color}
\usepackage{here}
 \makeatletter
    \@namedef{subjclassname@2020}{
  \textup{2020} Mathematics Subject Classification}
  \makeatother
\newtheorem{thm}{Theorem}[section]
\newtheorem{lem}[thm]{Lemma}

\newtheorem{cor}[thm]{Corollary}

{\theoremstyle{definition} 
\newtheorem{defn}[thm]{Definition}
\newtheorem{rem}[thm]{Remark}
\newtheorem{ex}[thm]{Example}}
\newtheorem{question}[thm]{Question}

\begin{document}
%
%
%
%
%
%
%
\title[]
{Knots with infinitely many non-characterizing slopes}
\author{Tetsuya Abe and Keiji Tagami}
\subjclass[2020]{57K10}
\keywords{annulus presentation, annulus twist, characterizing slope, Dehn surgery}
\address{Department of Mathematical Sciences,  Ritsumeikan University, Kusatsu-Shiga, Japan}
\email{tabe@fc.ritsumei.ac.jp}
\address{Department of Fisheries Distribution and Management, 
National Fisheries University, Shimonoseki, Yamaguchi 759-6595 
JAPAN
}
\email{tagami@fish-u.ac.jp}
\date{\today}
\maketitle
%
\begin{abstract}
Using the techniques on annulus twists, 
we observe  that  $6_3$  has infinitely many  non-characterizing slopes,
which affirmatively answers  a question  by Baker and Motegi.
Furthermore,  
 we prove that 
the knots $6_2$, $6_3$, $7_6$, $7_7$,  $8_1$, $8_3$,  $8_4$,  $8_6$,  $8_7$,  $8_9$,
$8_{10}$, $8_{11}$, $8_{12}$, $8_{13}$, $8_{14}$,
$8_{17}$, $8_{20}$ and $8_{21}$  have infinitely many non-characterizing slopes.
We also introduce the notion of 
trivial  annulus twists
 and give  some possible applications. Finally, we completely determine which knots have special
 annulus presentations up to 8-crossings.
\end{abstract}
\section{Introduction}
The classical theorem of Lickorish \cite{Lickorish} and Wallace \cite{Wallace} states that 
every closed, connected and orientable 3-manifold is obtained by Dehn surgery on a link in the 
$3$-sphere $\mathbf{S}^3$.
It is well-known that this surgery description of the manifold is far
from unique even if one restricts to knots. For example, 
the first  author, Jong, Luecke and  Osoinach   \cite{AJLO} proved that, for any integer $n$,
there exist infinitely many different knots in $\mathbf{S}^3$ such that $n$-surgery on those knots yields the same $3$-manifold.
If one  considers only specific surgery slopes
of a given knot, then the uniqueness problem becomes  meaningful.
A possible way to formulate the problem is via the  notion of a ``characterizing slope'' as follows.
%
%
\par
Let $K$ be a knot in $\mathbf{S}^3$.
We denote by $M_{K}(p/q)$ the $3$-manifold obtained from $\mathbf{S}^3$ by $p/q$-surgery 
on $K$. 
A slope $p/q \in \mathbf{Q}$ is \textit{characterizing} for $K$ if a knot $K'$ is isotopic to $K$ whenever $M_{K'}(p/q)$ is orientation-preservingly homeomorphic to $M_{K}(p/q)$. 
Using  monopole Floer homology, 
Kronheimer, Mrowka, Ozsv\'{a}th and Szab\'{o} \cite{KMOZ}  
proved that every non-trivial slope of the unknot
is characterizing, 
which was conjectured by Gordon  \cite{Gordon} in 1978.
Subsequently, Ozsv\'{a}th and Szab\'{o} \cite{OS} proved that 
 every non-trivial slope of the trefoil  and the figure eight knots
is characterizing using  Heegaard Floer homology. 
For more results, see \cite{McCoy,McCoy2,  NX}.
Recently,   Lackenby \cite{Lackenby}
proved that  if $|p| \le |q|$ and $|q|$ is sufficiently large, 
then  $p/q$ is characterizing.
\par
On the other hand,
for integral slopes, the situation is quite different.
Indeed, Baker and Motegi \cite{BM} proved the following.
\begin{thm}[{\cite[Theorem~1.5]{BM}}]\label{thm:main}
There exists a hyperbolic knot for which every integral slope is non-characterizing.
In particular, every integral slope of $8_6$ in Rolfsen's table is non-characterizing.
\end{thm}
\par
Baker and Motegi asked the following.
\par
\begin{question}$($\cite[Question 1.7]{BM}$)$ Are there any knots of crossing number less than 8 that have infinitely
many non-characterizing slopes?
\end{question}
\par
Using the techniques on annulus twists developed by
the first  author, Jong, Luecke and Osoinach  \cite{AJLO}, 
we affirmatively answer this question as follows.
\begin{thm}\label{thm:6_3}
The knot  $6_3$  has infinitely many  non-characterizing slopes.
\end{thm}
\par
One may consider that 
knots with infinitely many non-characterizing slopes are  sporadic.
In this paper,  we  prove the  following theorem, which suggests that 
such knots  are more common.
\begin{thm}\label{thm:8}
The following knots have infinitely many non-characterizing slopes:
\[ 
6_2,\ 6_3,\ 7_6,\ 7_7,\  8_1,\ 8_3,\  8_4,\ 8_6,\  8_7,\  8_9,\ 
8_{10},\ 8_{11},\ 8_{12},\ 8_{13},\ 8_{14},\ 8_{17},\ 8_{20},\ 8_{21}.
\]
\end{thm}
\par
We also propose a possible application to construct some interesting knots in a 3-manifold
as follows: 
In  knot theory, one of the basic questions is whether 
equivalent knots in a given $3$-manifold are isotopic or not.
Here, two knots $K_1$ and $K_2$ in an oriented $3$-manifold $M$ are \textit{equivalent} if there is an orientation-preserving homeomorphism $f\colon (M,K_1)\rightarrow (M,K_2)$.
It is well known that 
equivalent knots in $\mathbf{S}^3$ are isotopic.
On the other hand,  there exist 
equivalent knots in a 3-manifold which 
are not isotopic.
For more details, see \cite{CM}.
By considering a triviality of annulus twists, we 
 construct candidates of   such knots,  see Section~\ref{subsec:equivalent-knots}.
 Finally, we give a complete list of prime knots with special annulus presentations up to 8-crossings.


The rest of this paper is organized as follows:
In Section~\ref{sec:annulus}, 
we first recall the definition of (special) annulus presentations  of knots.
Next, we define the $n$-fold annulus twist along an annulus presentation, 
which is used to construct knots in $\mathbf{S}^3$ such that $0$-surgery on the knots yields
the same $3$-manifold.
In Section~\ref{sec:main}, we recall the definition of the operation $(\ast m)$,
which is used to construct knots in $\mathbf{S}^3$ such that $m$-surgery on the knots yields
the same $3$-manifold.
Using the operation $(\ast m)$, we prove  Theorem~\ref{thm:6_3}. 
In Section~\ref{sec:condition}, 
we recall  a simple sufficient condition for a given knot 
to have infinitely many non-characterizing slopes given by Baker and Motegi
 (Theorem~\ref{thm:BM}).
Using this sufficient condition, 
we prove  the main theorem (Theorem~\ref{thm:8}) in Section~\ref{sec:main-thm2}.
In Section~\ref{sec:trivial}, we introduce the notion of 
trivial  annulus twists and  investigate a sufficient condition for an annulus twist to be trivial.
As a byproduct, we  obtain a new method to construct 
equivalent knots in a 3-manifold which might be non-isotopic.
In Section~\ref{sec:table},
 we  tabulate  special annulus presentations of prime knots up to 8-crossings.
 We also introduce the notion of 
equivalent annulus presentations  and study its basic property.

%
\subsection{Notations}\label{notations}
Throughout this paper, 
\begin{itemize}
\item unless specifically mentioned, all knots and links are smooth and unoriented, and all other manifolds are smooth and oriented, 
\item for a $2$-component link $L_{1}\cup L_{2}$, we denote the $3$-manifold obtained from $\mathbf{S}^3$ by $m_1$-surgery on a knot $L_1$ and 
$m_2$-surgery on a knot $L_2$ by $M_{L_{1}\cup L_{2}}(m_1,m_2)$,
\item we denote the unknot in $\mathbf{S}^3$ by $U$, 
\item we denote a tubular neighborhood of a knot $K$ in a $3$-manifold by $\nu(K)$, 
\item we will use $\cong $ to denote orientation-preservingly diffeomorphic $4$-manifolds or homeomorphic $3$-manifolds. 
\end{itemize}
\section{Annulus twist along an annulus presentation}\label{sec:annulus}
In this section, 
we first recall the definition of annulus presentations  of knots, see  \cite{AJOT}, \cite[Section~5.3]{Abe-Tagami} and \cite{tagami-pre1}.
The important point is that, 
for a given annulus presentation  of a knot $K$,
we can find an embedded annulus $A'$ with a good property, 
which often intersects with $K$.
Next, using the embedded annulus $A'$, 
we define the $n$-fold annulus twist along an annulus presentation,
which is used to construct a new knot from $K$.
The goal of this section is to understand (the statement of) Theorem~\ref{thm:Osoinach}.
\subsection{Annulus presentation}
Fix a knot $K\subset \mathbf{S}^3$. 
Let $A\subset  \mathbf{S}^3$ be an embedded annulus, and take an embedding of a band $b\colon I\times I\rightarrow  \mathbf{S}^3$ such that 
\begin{itemize}
\item $b(I\times I)\cap \partial A=b(\partial I\times I)$, 
\item $b(I\times I)\cap \operatorname{Int} A$ consists of ribbon singularities, and 
\item $A\cup b(I\times I)$ is an immersion of an orientable surface, 
\end{itemize}
where $I$ is the unit interval $[0, 1]$, see  Figure 1(center).
We call the pair $(A, b)$ an {\it annulus presentation} of $K$ if the knot $K$ is isotopic to  $(\partial A\setminus b(\partial I\times I))\cup b(I\times \partial I)$, see Figure 1(left and center).
\par
In this paper, 
we mainly consider 
a particular type of  annulus presentation as follows:
An annulus presentation $(A,b)$ is {\it special} if the embedded  annulus $A$ is  a Hopf band.
A special annulus presentation $(A,b)$ is {\it positive} if 
$A$ is the positive Hopf band, and  {\it negative} if $A$ is the negative Hopf band.
It is easily seen that a knot has a positive (resp.~negative) special annulus presentation if and only if $K$ is obtained from the positive (resp.~negative) Hopf link by a single coherent band surgery
after giving some orientation to $K$. 
\begin{ex}\label{ex:6_3}
The knot $6_{3}$ has a negative  special annulus presentation $(A, b)$ 
as in Figure~\ref{figure:annulus-pre}. 
If we set $\partial A=c_{1}\cup c_{2}$
as in Figure~\ref{figure:annulus-pre} and 
give a parallel orientation to $c_{1}\cup c_{2}$, 
then  lk$(c_1,c_2)=1$, 
where lk is the linking number.
Note that $A$ is the negative Hopf band, however,  we have lk$(c_1,c_2)=1$.
\end{ex}
\begin{figure}[h!]
\centering \includegraphics[scale=0.75]{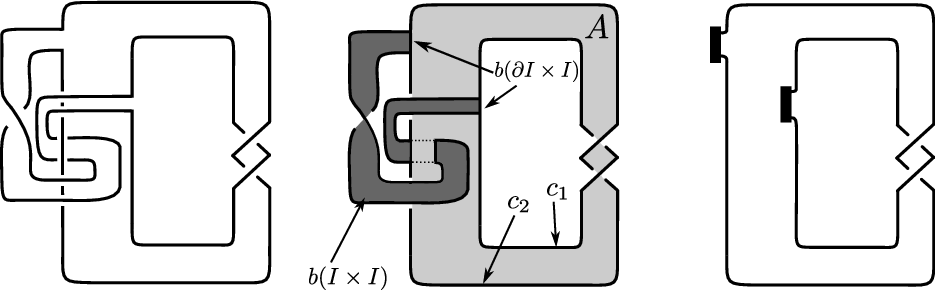}
\caption{The knot $6_3$ (left). 
A negative special annulus presentation $(A, b)$ of $6_3$ (center). 
For simplicity, 
we often draw the attaching regions for the band by bold arcs and the bands are omitted (right). }
\label{figure:annulus-pre}
\end{figure}

\begin{rem}
A knot $K$ has an annulus presentation if and only if 
the mirror image $\overline{K}$  has an annulus presentation.
This fact is implicitly  used in Section~\ref{sec:table}.
\end{rem}
\subsection{Annulus twist along an annulus presentation}\label{sub:Annulus twist}
Let $K$ be a knot with an annulus presentation $(A, b)$,
which may not be special. 
We order the two components of the boundary $\partial A$, and set $\partial A=c_{1}\cup c_{2}$
with a parallel orientation.
Then we can find a ``shrunken'' annulus $A' \subset A$ with
(ordered)  boundary $\partial A'=c'_{1}\cup c'_{2}$ which satisfies the following: 
\begin{itemize}
\item The closure of ${A\setminus A'}$, denoted by $\overline{A\setminus A'}$,  is a disjoint union of two annuli, 
\item each $c'_{i}$ ($i=1,2$) is isotopic to $c_{i}$ in 
$\overline{A\setminus A'}$,  and 
\item $A\setminus (\partial A\cup A')$ does not intersect $b(I\times I)$. 
\end{itemize}
The left picture in Figure~\ref{figure:annulus-twist} will help us to understand the definition of $A'$.
Using the embedded annulus $A'$,  we define as follows:
\begin{defn}
Let $n$ be an  integer. 
{\it The  $n$-fold annulus twist along $(A,b)$}
is to  apply $(lk(c_1,c_2)+1/n)$-surgery on $c'_{1}$ and $(lk(c_1,c_2)-1/n)$-surgery on $c'_{2}$.
\end{defn}
Note that the surgered 3-manifold
$M_{c'_{1}\cup c'_{2}}(lk(c_1,c_2)+1/n, lk(c_1,c_2)-1/n)$ is homeomorphic to $\mathbf{S}^3$,
 see \cite[Theorem~2.1]{Osoinach}.
Since $(c'_1 \cup c'_2) \cap K = \emptyset$,  for a given integer $n$, 
we obtain a new knot $A^{n}(K)$ from  $K$ 
by the $n$-fold annulus twist along $(A,b)$.
We call $A^{n}(K)$  the knot obtained from $K$ by {\it the $n$-fold annulus twist along $(A,b)$}.
The knot $A^{1}(K)$  is called the knot obtained from $K$ by the annulus twist along $(A,b)$ and 
we denote it by $A(K)$. 


%

%
\begin{ex}
Let $(A, b)$ be the special annulus presentation of $6_3$ given in Figure~\ref{figure:annulus-pre}.
Then  $A(6_{3})$ is the knot in Figure~\ref{figure:annulus-twist} (center and right).
\end{ex}
\begin{figure}[h]
\centering
\includegraphics[scale=0.85]{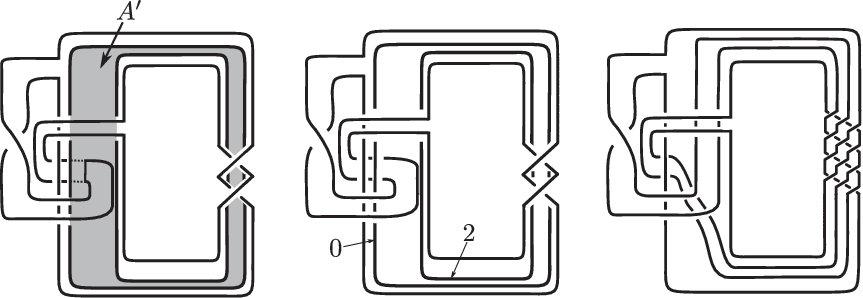}
\caption{A shrunken annulus $A'$ (left),  
 $A(6_{3})$ in the surgered 3-manifold $M_{c'_{1}\cup c'_{2}}(2,0)$ (center) and
 $A(6_{3})$ in $\mathbf{S}^3$ (right) 
}\label{figure:annulus-twist}
\end{figure}
The following is essentially due to Osoinach~\cite[Theorem~2.3]{Osoinach} (see also \cite{Teragaito}). 
\begin{thm}\label{thm:Osoinach} 
Let $K\subset \mathbf{S}^3$ be a knot with an annulus presentation $(A, b)$, 
which may not be special. 
Then, there is an  orientation-preservingly homeomorphism $\phi_{n}\colon M_{K}(0) \rightarrow M_{A^{n}(K)}(0)$ for any $n\in\mathbf{Z}$. 
In particular, $\phi_{n}$ is given as in Figure~\ref{figure:Osoinach-homeo}. 
\end{thm}
\begin{figure}[h!]
\centering
\includegraphics[scale=0.8]{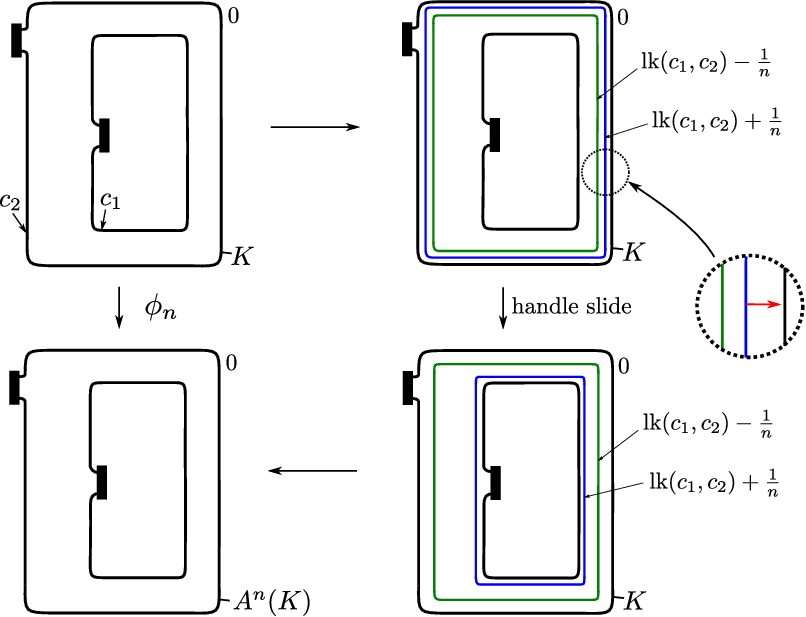}
\caption{
(color online) The definition of the $n$-th Osoinach-Teragaito's homeomorphism $\phi_{n}$ }
\label{figure:Osoinach-homeo}
\end{figure}
We call $\phi_{n}$ the $n$-th Osoinach-Teragaito's homeomorphism 
since Osoinach \cite{Osoinach} introduced the homeomorphism $\phi_{n}$    and
Teragaito \cite{Teragaito} gave  a surgery description of $\phi_{n}$.
We  simply 
call $\phi_{1}$ Osoinach-Teragaito's homeomorphism  and 
denote it by $\phi$. 
\begin{ex}\label{ex:6_3con}
We have $M_{6_3}(0)  \cong  M_{A^{n}(6_3)}(0)$ for any integer $n$ by Theorem~\ref{thm:Osoinach}.
\end{ex}
%
%
\section{Proof of Theorem~\ref{thm:6_3}}\label{sec:main}
The first  author, Jong, Luecke and Osoinach \cite[Section~3.1.2]{AJLO} introduced the operation $(\ast m)$ to construct infinitely many distinct knots with the same $m$-surgery. 
In this section, we recall the definition of the operation $(\ast m)$ and 
highlight a property of the operation $(\ast m)$.
As an  application, we prove Theorem~\ref{thm:6_3},
which states that  $6_3$  has infinitely many  non-characterizing slopes.

Let $K$ be a knot with a special annulus presentation $(A,b)$.
We define  the operation $(\ast m)$  as follows:
Let $A(K)$ be the knot obtained from $K$ by the annulus twist along $(A,b)$, and
$\gamma_{A(K)}$  a curve in $\mathbf{S}^3\setminus \nu(A(K))$ depicted in Figure~\ref{figure:gamma}.
\begin{figure}[h]
\centering
\includegraphics[scale=0.8]{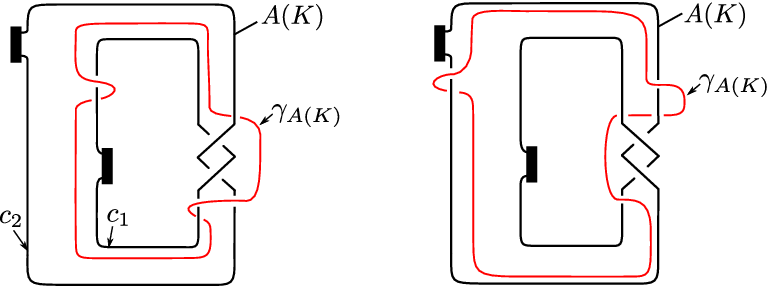}
\caption{(color online) The definition of the curve $\gamma_{A(K)}$
}\label{figure:gamma}
\end{figure}
Note that the definition of  $\gamma_{A(K)}$ depends  on the twist of $A$.
We denote by $T_{m}(A(K))$ the knot obtained from $A(K)$ by twisting $m$ times along $\gamma_{A(K)}$. 
The operation $K\mapsto T_{m}(A(K))$ is called {\it the operation $(\ast m)$}. 
The most important property of the operation $(\ast m)$ is the following.
\begin{thm}[{\cite[Theorem~3.7]{AJLO}}]\label{thm:AJLO1}
Let $K$ be a knot with a special annulus presentation $(A,b)$. 
Then, there is an orientation-preservingly homeomorphism $\psi_{m}\colon M_{K}(m)\rightarrow M_{T_{m}(A(K))}(m)$.
\end{thm}
{In \cite[Figure 15]{AJLO}, we find a proof of Theorem 3.1 for the case where $A$ is the negative
Hopf band. For the reader's convenience, in Appendix, we give
a complete proof of Theorem 3.1. }

\begin{rem}
In Theorem 3.1, we {cannot} remove the assumption  that an annulus presentation $(A, b)$ is special.
If the annulus $A$ is knotted, then the corresponding curve $\gamma_{A(K)}$ will be knotted.
Even if the annulus $A$ is unknotted, if $A$ is not Hopf bands, the corresponding {slope} of  the curve $\gamma_{A(K)}$ will not be  $-\dfrac{1}{m}$.
\end{rem}

%
%
%
\par
Now we are ready to prove Theorem~\ref{thm:6_3}.
\begin{proof}[Proof of Theorem~\ref{thm:6_3}]
Let $(A, b)$ be the special annulus presentation of $6_3$ given in Figure~\ref{figure:annulus-pre}.
By Theorem~\ref{thm:AJLO1}, we have $M_{6_3}(m)  \cong  M_{T_{m}(A(6_3))}(m)$.
All we have to prove is  that
\[ 6_3 \neq T_{m}(A(6_3))\]
for infinitely many integers $m$. By \cite[Lemma~3.12]{AJLO},
we have $\Delta_{6_3}(t) \neq \Delta_{T_{m}(A(6_3))}(t)$ for any positive integer $m$,
where $\Delta_{K}(t)$ is the Alexander polynomial of a knot $K$.
This implies that $6_3 \neq T_{m}(A(6_3))$ for infinitely many integers $m$,
that is, $6_3$  has infinitely many  non-characterizing slopes. 
\end{proof}
\section{Baker-Motegi's condition on  non-characterizing slopes}\label{sec:condition}
In \cite{BM}, Baker and Motegi gave a sufficient condition for a given knot 
to  have infinitely many non-characterizing slopes.
In this section, we show  that a knot which has a special annulus presentation with some property satisfies
Baker-Motegi's condition.

It seems to be convenient to  introduce the following terminology.
\begin{defn}
A knot  $K \subset \mathbf{S}^3$ 
satisfies {\it Baker-Motegi's condition} (BM-condition for short) if we can take an unknot $c$ disjoint from $K$ which satisfies the following properties:
\begin{itemize}
\item $c$ is not a meridian of $K$, and
\item the $(0,0)$-surgery on $K \cup c$ yields $\mathbf{S}^3$.
\end{itemize}
\end{defn}
Under this terminology, Baker and Motegi proved the following theorem.

\begin{thm}$($\cite[Theorem~1.3]{BM}$)$ \label{thm:BM}
Any knot with BM-condition has infinitely many non-characterizing slopes.
\end{thm}

\begin{rem}
The technique used in \cite{BM}  is essentially the same as use of a dualizable pattern defined by Gompf-Miyazaki \cite{Gompf-Miyazaki} and developed by Miller and Piccirillo \cite{Miller-Piccirillo}. 
\end{rem}

Baker and Motegi asked which knots satisfy BM-condition, 
see \cite[Question~5.3]{BM}.
They proved that any L-space knot does not satisfy BM-condition, which is a ``negative''  result (\cite[Theorem~1.8]{BM}).
The following lemma gives a constructive result.

\par
\begin{lem}\label{lem:non-trivial}
Let $K$ be a knot with a special annulus presentation $(A,b)$ 
and $\beta_{K}\subset\mathbf{S}^3\setminus \nu(K)$ the closed curve depicted in Figure~\ref{figure:beta}. 
Then the $(0,0)$-surgery on $K \cup \beta_{K}$ yields $\mathbf{S}^3$.
Moreover, if $\beta_{K}$ is not a meridian of $K$ in $\mathbf{S}^3$, then 
$K$ satisfies BM-condition.
\end{lem}
\begin{proof}
\begin{figure}[h!]
\centering
\includegraphics[scale=0.8]{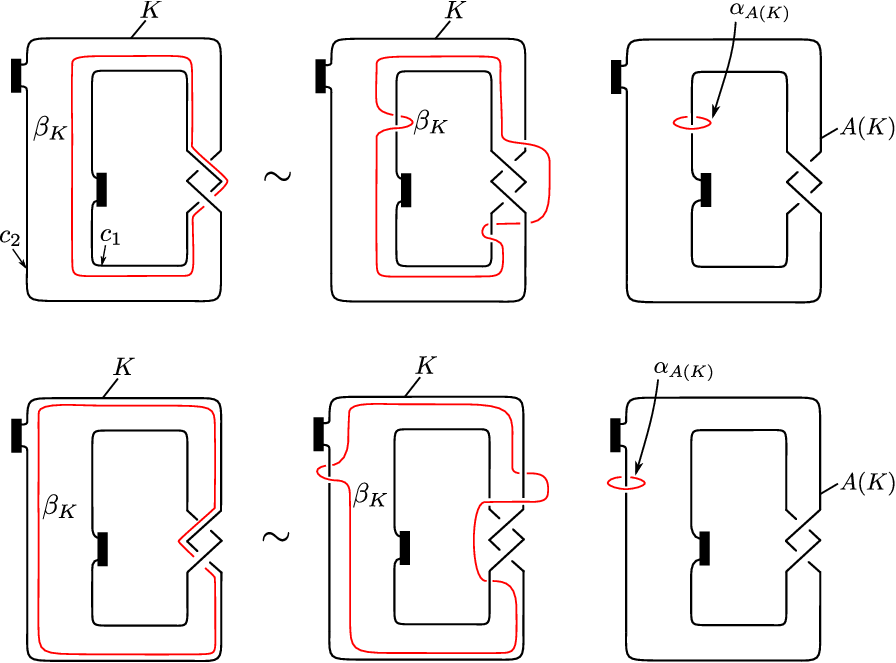}
\caption{(color online) The definition of the closed curve  $\beta_{K}\subset\mathbf{S}^3\setminus \nu(K)$}
\label{figure:beta}
\end{figure}
Let $L\subset M_{K}(0)$ be the surgery dual to $K$. 
We can regard $\beta_{K}$ as a knot in $M_{K}(0)$ because of $\beta_{K}\subset \mathbf{S}^3\setminus \nu(K)=M_{K}(0)\setminus \nu(L)$. 
Then, we can check that
 Osoinach-Teragaito's homeomorphism $\phi$ induces a homeomorphism 
 \[ \phi\colon (M_{K}(0),\beta_{K})\rightarrow (M_{A(K)}(0),\alpha_{A(K)}),\]
 where  $\alpha_{A(K)}$ is a meridian of $A(K)$. 
Moreover, we see that $\phi$ preserves {slopes} of $\beta_{K}$ and $\alpha_{A(K)}$. 
Since $M_{A(K)\cup \alpha_{A(K)}}(0,0)\cong \mathbf{S}^3$, we have $M_{K\cup \beta_{K}}(0,0)\cong \mathbf{S}^3$. 
Moreover, if $\beta_{K}$ is not a meridian of $K$, we see that $K$ satisfies BM-condition by taking $c=\beta_{K}$ in the definition of BM-condition. 
\end{proof}
\begin{rem}
There is an alternative proof of Theorem~\ref{thm:6_3}.
By applying Lemma~\ref{lem:non-trivial} to the special annulus  presentation of $6_3$ 
depicted in Figure~\ref{figure:annulus-pre}, we see that $6_3$  has infinitely many non-characterizing slopes.
For the detail, see the next section.
In the forthcoming paper, we prove that the two proofs of Theorem~\ref{thm:6_3} are essentially the same.
\end{rem}
%
%

%
\section{Proof of Theorem~\ref{thm:8}}\label{sec:main-thm2}
In this section, we prove  Theorem~\ref{thm:8}, which states that 
the knots $6_2$, $6_3$, $7_6$, $7_7$,  $8_1$, $8_3$, $8_4$, $8_6$, $8_7$, 
$8_9$, $8_{10}$, $8_{11}$, $8_{12}$, $8_{13}$, $8_{14}$, $8_{17}$, $8_{20}$ and $8_{21}$ 
have infinitely many non-characterizing slopes.

\par
The following lemma is useful in finding a special annulus presentation of  a given knot.
\begin{lem}$($\cite[Lemma~2.2]{AJOT}$)$ \label{lem:unknotting}
Any unknotting number one knot has a special annulus presentation.
\end{lem}

\begin{rem}
The unknotting number of the Conway knot  is one.
In \cite{Piccirillo}, Piccirillo used  Lemma~\ref{lem:unknotting} implicitly, see Figure~\ref{figure:Conway} 
(compare with (d) in Figure~4 in \cite{Piccirillo}).
\begin{figure}[h!]
\centering
\includegraphics[scale=0.1]{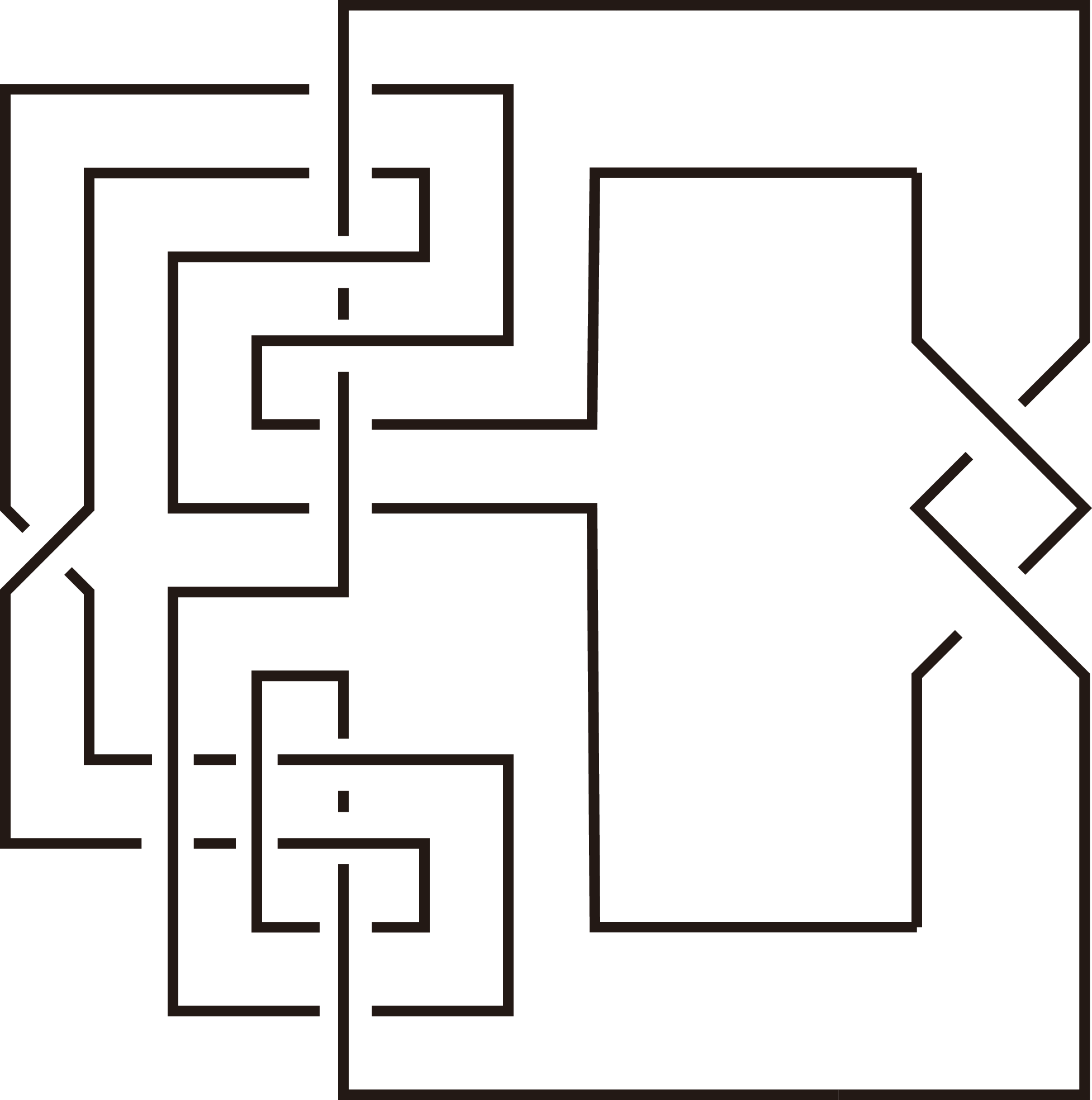}
\caption{A special annulus presentation $(A, b)$ of the Conway knot}
\label{figure:Conway}
\end{figure}
\end{rem}
\par
We obtain the following.
\begin{lem} \label{lem:annulus-pre} The following knots have special annulus presentations:
$U$, $3_1$, $4_1$, $5_2$, $6_1$, $6_2$, $6_3$, $7_2$, $7_6$, $7_7$, 
$8_1$, $8_3$, $8_4$, $8_6$, $8_7$, $8_9$, $8_{10}$, $8_{11}$, $8_{12}$, 
$8_{13}$, $8_{14}$, $8_{17}$, $8_{20}$, $8_{21}$.
\end{lem}
\begin{proof}
The unknotting number of $U$, $3_1$, $4_1$, $5_2$, $6_1$, $6_2$, $6_3$, $7_2$, $7_6$, $7_7$, 
$8_1$, $8_7$, $8_9$, $8_{11}$, $8_{13}$, $8_{14}$, $8_{17}$, $8_{20}$ and $8_{21}$ is one. By Lemma~\ref{lem:unknotting}, 
these knots have special annulus presentations.
The unknotting number of $8_3,  8_4, 8_6,8_{10}$ and $8_{12}$  is two. 
We can find  special annulus presentations for these knots, see Figure~\ref{annulus-pre}.
\end{proof}
\begin{figure}[h]
\begin{center}
\includegraphics[scale=0.7]{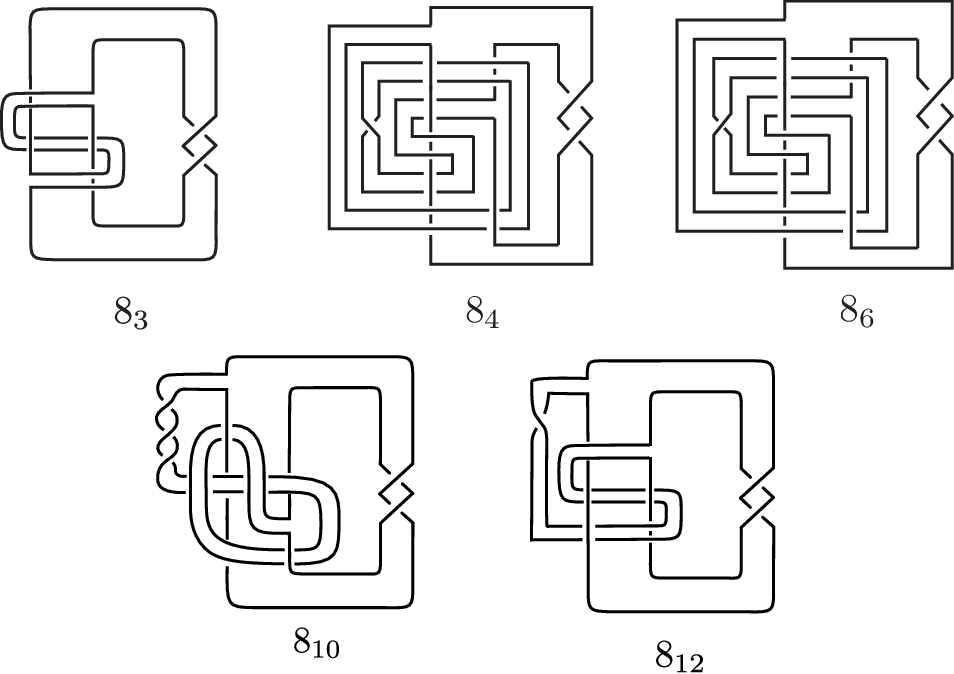}
\end{center}
\caption{Special annulus presentations of $8_3$, $8_4$, $8_6$, $8_{10}$ and $8_{12}$}
\label{annulus-pre}
\end{figure}
The following lemma is a key to prove Theorem~\ref{thm:8}.
\begin{lem} \label{lem:dualizable}
The  knots
$6_2$, $6_3$, $7_6$, $7_7$, $8_1$, $8_3$, $8_4$, $8_6$, $8_7$, $8_9$, $8_{10}$, 
$8_{11}$, $8_{12}$, $8_{13}$, $8_{14}$, $8_{17}$, $8_{20}$ and $8_{21}$ satisfy 
BM-condition. 
\end{lem}

\begin{proof}
Let $\kappa$ be a knot with a special annulus presentation.
We denote by $\alpha_{\kappa}$    a meridian of $\kappa$ and  
by $\beta_{\kappa}$  the closed curve defined in Figure~\ref{figure:beta}.
We first consider $6_2$.
It has a special annulus presentation and the  link $6_2 \cup \beta_{6_2}$ is
 the left picture in Figure~\ref{6_2}.
Then  we have
\[\nabla_{6_2 \cup \alpha_{6_2}}(z)= 
-z^{5}-z^{3}+z,\ \ \ \   \nabla_{6_2 \cup \beta_{6_2}}(z)= 
-z^{3}+z, \]
where we choose orientations of $6_2  \cup \alpha_{6_2} $ and  $6_2  \cup \beta_{6_2} $
so that the linking numbers of $6_2  \cup \alpha_{6_2} $ and $6_2  \cup \beta_{6_2} $ 
are one, respectively.  Here $\nabla_{L}(z)$ is the Conway polynomial of an oriented link $L$.
Therefore  $\beta_{6_2}$ is not a meridian of $6_2$. 
This implies that $6_2$ satisfies 
BM-condition by Lemma~\ref{lem:non-trivial}.

Let $K$ be one of the knots  $6_3$, $7_6$, $7_7$, $8_1$, $8_3$, $8_4$, $8_6$, $8_7$, $8_9$,
$8_{10}$, $8_{11}$, $8_{12}$, $8_{13}$, $8_{14}$, $8_{17}$, $8_{20}$ and $8_{21}$.
It has a special annulus presentation and the  link $K \cup \beta_{K}$ is as 
 in Figures~\ref{6_2} and \ref{7_7}.
 Then  we have
$\nabla_{K\cup \alpha_{K}}(z)
 \neq \nabla_{K \cup \beta_{K}}(z)$, 
where we choose orientations of $K \cup \alpha_{K}$ and  $K  \cup \beta_{K} $
so that their  linking numbers are one. 
Therefore   $\beta_{K}$ is not a meridian of $K$. 
This implies that $K$ satisfies 
BM-condition by Lemma~\ref{lem:non-trivial}.
 For the actual calculations, see the following.
\begin{align*}
\nabla_{6_3 \cup \alpha_{6_3}}(z)&=z^{5}+z^{3}+z,  &\nabla_{6_3 \cup \beta_{6_3}}(z)&=z^{3}+z,  \\
\nabla_{7_6 \cup \alpha_{7_6}}(z)&=-z^{5}+z^{3}+z,  &\nabla_{7_6 \cup \beta_{7_6}}(z)&=z^{3}+z,  \\
\nabla_{7_7 \cup \alpha_{7_7}}(z)&=z^{5}-z^{3}+z,&   \nabla_{7_7 \cup \beta_{7_7}}(z)&=-z^{3}+z,\\ 
\nabla_{8_1 \cup \alpha_{8_1}}(z)&=-3z^{3}+z,&\nabla_{8_1 \cup \beta_{8_1}}(z)&=-z^{5}-3z^{3}+z, \\
\nabla_{8_3 \cup \alpha_{8_3}}(z)&= -4z^{3}+z,&\nabla_{8_3 \cup \beta_{8_3}}(z)&= -z^{5}-4z^{3}+z, \\
\nabla_{8_{4} \cup \alpha_{8_{4}}}(z)&=-2z^{5}-3z^{3}+z,& \nabla_{8_{4} \cup \beta_{8_{4}}}(z)&=-z^{5}-3z^{3}+z, \\
\nabla_{8_{6} \cup \alpha_{8_{6}}}(z)&=-2z^{5}-2z^{3}+z,& \nabla_{8_{6} \cup \beta_{8_{6}}}(z)&=-z^{5}-2z^{3}+z, \\
\nabla_{8_7 \cup \alpha_{8_7}}(z)&= z^{7}+3z^{5}+2z^{3}+z,& \nabla_{8_7 \cup \beta_{8_7}}(z)&= 2z^{3}+z, \\
\nabla_{8_9 \cup \alpha_{8_9}}(z)&= -z^{7}-3z^{5}-2z^{3}+z,&\nabla_{8_9 \cup \beta_{8_9}}(z)&= -2z^{3}+z, \\
\nabla_{8_{10} \cup \alpha_{8_{10}}}(z)&=z^{7}+3z^{5}+3z^{3}+z,& \nabla_{8_{10} \cup \beta_{8_{10}}}(z)&=3z^{3}+z, \\
\nabla_{8_{11} \cup \alpha_{8_{11}}}(z)&= -2z^{5}-z^{3}+z,&  \nabla_{8_{11} \cup \beta_{8_{11}}}(z)&= -z^{3}+z, \\
\nabla_{8_{12} \cup \alpha_{8_{12}}}(z)&=z^{5}-3z^{3}+z,& \nabla_{8_{12} \cup \beta_{8_{12}}}(z)&=z^{7}+3z^{5}-3z^{3}+z, \\
\nabla_{8_{13} \cup \alpha_{8_{13}}}(z)&=2z^{5}+z^{3}+z,& \nabla_{8_{13} \cup \beta_{8_{13}}}(z)&=z^{3}+z, \\
\nabla_{8_{14} \cup \alpha_{8_{14}}}(z)&=-2z^{5}+z,& \nabla_{8_{14} \cup \beta_{8_{14}}}(z)&=z, \\
\end{align*}
\begin{align*}
\nabla_{8_{17} \cup \alpha_{8_{17}}}(z)&=-z^{7}-2z^{5}-z^{3}+z,&\nabla_{8_{17} \cup \beta_{8_{17}}}(z)&=-z^{3}+z, \\
\nabla_{8_{20} \cup \alpha_{8_{20}}}(z)&= z^{5}+2z^{3}+z,&  \nabla_{8_{20} \cup \beta_{8_{20}}}(z)&= 2z^{3}+z, \\
\nabla_{8_{21} \cup \alpha_{8_{21}}}(z)&= -z^{5}-3z^{3}+z,&  \nabla_{8_{21} \cup \beta_{8_{21}}}(z)&=-3z^{3}+z.
\end{align*}
\end{proof}

\begin{figure}[h!]
\centering
\includegraphics[scale=0.12]{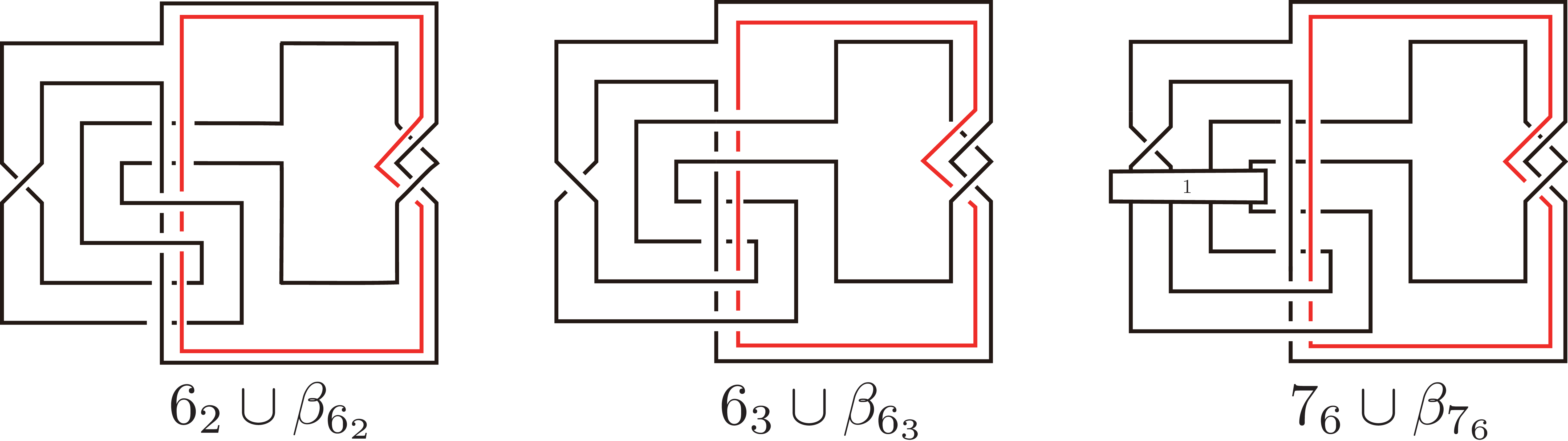}
\caption{(color online) 2-component links $6_2 \cup \beta_{6_2}$, $6_3 \cup \beta_{6_3}$ and $7_6 \cup \beta_{7_6}$}
\label{6_2}
\end{figure}
\begin{figure}[h!]
\centering
\includegraphics[scale=0.1]{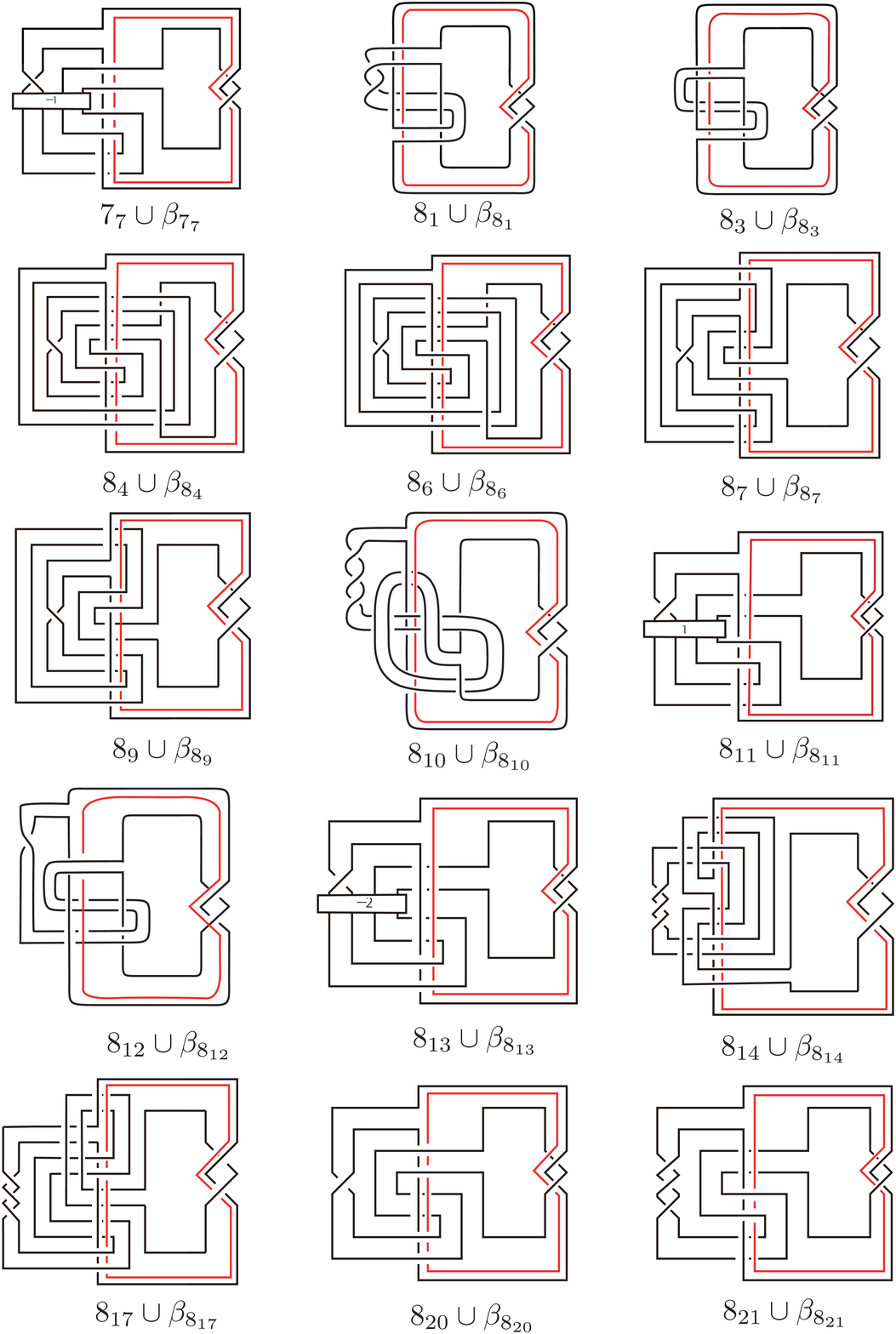}
\caption{(color online) 2-component links $7_{7} \cup \beta_{7_7}$, 
$8_{1} \cup \beta_{8_{1}}$, 
$8_{3} \cup \beta_{8_{3}}$, 
$8_{4} \cup \beta_{8_{4}}$, 
$8_{6} \cup \beta_{8_{6}}$, 
$8_{7} \cup \beta_{8_{7}}$, 
$8_{9} \cup \beta_{8_{9}}$, 
$8_{10} \cup \beta_{8_{10}}$, 
$8_{11} \cup \beta_{8_{11}}$, 
$8_{12} \cup \beta_{8_{12}}$, 
$8_{13} \cup \beta_{8_{13}}$, 
$8_{14} \cup \beta_{8_{14}}$, 
$8_{17} \cup \beta_{8_{17}}$, 
$8_{20} \cup \beta_{8_{20}}$ and
$8_{21} \cup \beta_{8_{21}}$.
}
\label{7_7}
\end{figure}

We are ready to prove Theorem~\ref{thm:8}.
\begin{proof}[Proof of Theorem~\ref{thm:8}]
By Lemma~\ref{lem:dualizable}, the  knots
$ 6_2$, $6_3$, $7_6$, $7_7$, $8_1$, $8_3$,  $8_4$, $8_6$,  $8_7$,  $8_9$,
$8_{10}$, $8_{11}$, $8_{12}$, $8_{13}$, $8_{14}$, $8_{17}$, $8_{20}$ and $8_{21}$ satisfy 
BM-condition.
Therefore, these knots have infinitely many non-characterizing slopes  by Theorem~\ref{thm:BM}.
\end{proof}


%
\section{Triviality of annulus twists and  application to knot theory}\label{sec:trivial}
In this  section,  
we introduce the notion of 
trivial  annulus twists and investigate a sufficient condition for an annulus twist to be trivial.
As a byproduct, we  obtain a new method to construct 
equivalent knots in a 3-manifold which might be non-isotopic.
Throughout this section, we only consider $0$-surgery on knots.

\subsection{Triviality of the  $n$-fold annulus twist}
We introduce the notion of ``trivial" for the  $n$-fold annulus twist along an annulus presentation.
\par 
Let $K$ be a knot with an annulus presentation $(A,b)$
and  $A^{n}(K)$ the  knot  as in Section \ref{sub:Annulus twist}.
We denote by $\alpha_{A^{n}(K)} \subset \mathbf{S}^3$ a meridian of $A^{n}(K)$ and by $L_{A^{n}(K)}$
  the surgery dual to $A^{n}(K)$ in  $M_{A^{n}(K)}(0)$.
 We can regard $\alpha_{A^{n}(K)}$ as a curve in $M_{A^{n}(K)}(0)$ since  we have
\begin{align*}
\alpha_{A^{n}(K)} \subset \mathbf{S}^3\setminus \nu(A^{n}(K))= M_{A^{n}(K)}(0)\setminus \nu(L_{A^{n}(K)}).\label{eq:identification}
\end{align*}
\begin{defn}
The $n$-fold annulus twist along $(A,b)$ is {\it trivial} if $\phi_{n}(\alpha_{K})$ is isotopic to $\alpha_{A^{n}(K)}$ in $M_{A^{n}(K)}(0)$, where $\phi_{n}\colon M_{K}(0) \rightarrow M_{A^{n}(K)}(0)$ is the $n$-th Osoinach-Teragaito's homeomorphism. 
\end{defn}
\par
The following lemma justifies the above definition.
\begin{lem}\label{lem:trivial1}
Let $K$ be a knot with an annulus presentation $(A,b)$. 
Suppose that the $n$-fold annulus twist along $(A,b)$ is trivial. 
Then we obtain $K=A^{n}(K)$. 
\end{lem}
\begin{proof}
Since the $n$-fold annulus twist along $(A,b)$ is trivial, $\phi_{n}(\alpha_{K})$ is isotopic to $\alpha_{A^{n}(K)}$ in $M_{A^{n}(K)}(0)$. 
Note that $\alpha_{A^{n}(K)}$ is isotopic to $L_{A^{n}(K)}$ in $M_{A^{n}(K)}(0)$. 
Hence, we have 
\begin{align*}
\mathbf{S}^3\setminus \nu(K)
&=M_{K}(0)\setminus \nu(L_{K})\\
&\cong M_{K}(0)\setminus \nu(\alpha_{K})\\
&\cong \phi_{n}(M_{K}(0))\setminus \phi_{n}(\nu(\alpha_{K}))\\
&\cong M_{A^{n}(K)}(0)\setminus \nu(\alpha_{A^{n}(K)})\\
&\cong M_{A^{n}(K)}(0)\setminus \nu(L_{A^{n}(K)})\\
&=\mathbf{S}^3\setminus \nu(A^{n}(K)).
\end{align*}
By the Knot Complement Theorem~\cite{Gordon-Luecke}, we obtain $K=A^{n}(K)$. 
\end{proof}
The following lemma gives us many examples of 
trivial annulus twists. 

\begin{lem}\label{lem:trivial2}
Let $K$ be a knot with a special annulus presentation $(A,b)$. 
Suppose that there is some disk $D$ such that $\partial D=c_1$ and $\operatorname{Int} D \cap b=\emptyset$ (see Figure~\ref{figure:trivial-twist1}).  
Then, we obtain the following. 
\begin{itemize}
\item If $(A, b)$ is negative, the $(-1)$-fold annulus twist along $(A,b)$ is trivial. 
\item If $(A, b)$ is positive, the $(+1)$-fold annulus twist along $(A,b)$ is trivial. 
\end{itemize}

\end{lem}

\begin{proof}
We can assume that $\operatorname{Int} D \cap K$ consists of a single point as in Figure~\ref{figure:trivial-twist1}.
Suppose that $(A, b)$ is negative. 
Then, $(\phi_{-1})^{-1}\colon M_{A^{-1}(K)}(0)\rightarrow M_{K}(0)$ sends $\alpha_{A^{-1}(K)}$ to $\alpha_{K}$ (see Figure~\ref{figure:trivial-twist2}). 
Indeed, by a small isotopy, $(\phi_{-1})^{-1}(\alpha_{A^{-1}(K)})$ is isotopic to a curve $\ell \subset D$
and it is isotopic to $\alpha_{K}$ since we can find an annulus bounded by the curve $\ell$ and $\alpha_{K}$ by using the given disk $D$.
In the case   where $(A, b)$ is positive,  
we can prove that the $(+1)$-fold annulus twist along $(A,b)$ is trivial similarly. 
\end{proof}
\begin{figure}[h!]
\centering
\includegraphics[scale=0.75]{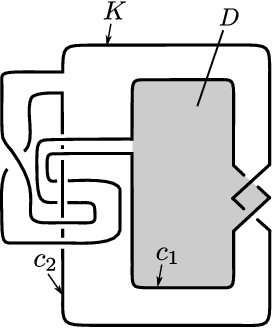}
\caption{A disk $D$ satisfying $\partial D=c_1$ and $\operatorname{Int} D \cap b=\emptyset$}
\label{figure:trivial-twist1}
\end{figure}
\begin{figure}[h!]
\centering
\includegraphics[scale=0.8]{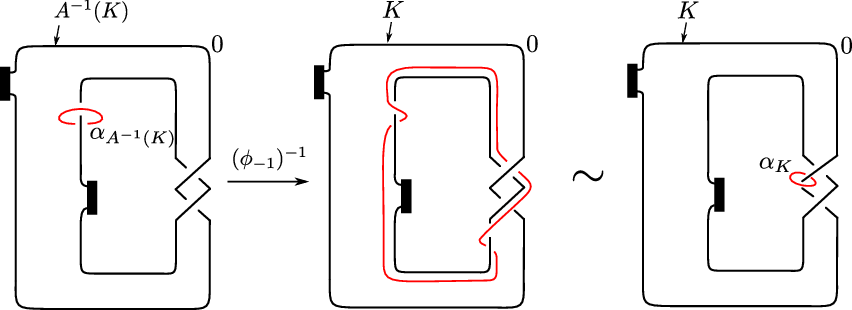}
\caption{(color online) Proof of Lemma~\ref{lem:trivial2}}
\label{figure:trivial-twist2}
\end{figure}
\par
As a corollary of Lemma~\ref{lem:trivial2}, 
we obtain the following, which gives an alternative proof of \cite[Theorem~5.1]{tagami-pre1} for the unoriented case. 
\begin{cor}[{e.g.~\cite[Theorem~5.1]{tagami-pre1}}]\label{cor:trivial2}
Let $K$ be a knot with a special annulus presentation $(A,b)$. 
Suppose that there is some disk $D$ such that $\partial D=c_1$ and $\operatorname{Int} D \cap b=\emptyset$ (see Figure~\ref{figure:trivial-twist1}).  
Then, we obtain the following. 
\begin{itemize}
\item If $(A, b)$ is negative, then $K=A^{-1}(K)$. 
\item If $(A, b)$ is positive, then  $K=A^{+1}(K)$. 
\end{itemize}
\end{cor}
\begin{ex}
Let $(A, b)$ be the special annulus presentation of $6_3$ as in Figure~\ref{figure:annulus-pre}.
Then  we have $6_3=A^{-1}(6_3)$ by Corollary \ref{cor:trivial2}.
\end{ex}
\subsection{Equivalent knots in a 3-manifold which might be non-isotopic}\label{subsec:equivalent-knots}
In this subsection, we consider whether the converse of Lemma~\ref{lem:trivial1} holds or not,
which motivates the following  question.
\begin{question}
Let $K$ be a knot with an annulus presentation $(A,b)$.
Fix  an integer  $n$.
If we have $K=A^{n}(K)$, 
 is the $n$-fold annulus twist along $(A,b)$ trivial? 
More strongly, 
is the $n$-th Osoinach-Teragaito's homeomorphism
\[  \phi_{n}\colon M_{K}(0) \rightarrow  M_{A^{n}(K)}(0)=M_{K}(0)\]
 is isotopic to the identity?
\end{question}
We do not have any counterexamples to this question at the time of writing. 
Potential counterexamples are constructed as follows:
Let $K$ be a knot with  an annulus presentation  $(A,b)$. 
Suppose that $\operatorname{Int}A\cap b =\emptyset$. 
Then, for any integer $n$,  we have $K=A^{n}(K)$.
However, we do not know whether the $n$-fold annulus twist along $(A,b)$ is trivial or not.
For example, the knot $8_{1}$ has such an annulus presentation, see Figure~\ref{equivalent-knots}.
If the $n$-fold annulus twist along the annulus presentation are not trivial, 
then the two knots $\alpha_{8_1}$ and $\phi(\alpha_{8_1})$ in $M_{8_1}(0)$ are not isotopic although they are equivalent. 
%
\begin{figure}[h!]
\begin{center}
\includegraphics[scale=0.8]{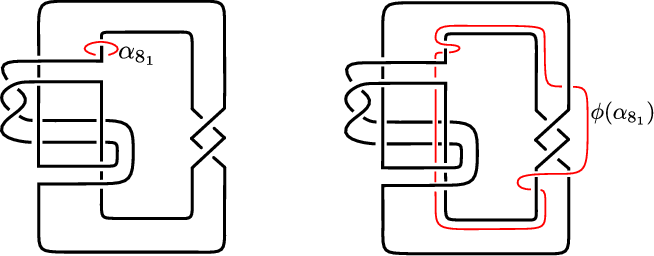}
\end{center}
\caption{(color online) Equivalent knots $\alpha_{8_{1}}$ and $\phi(\alpha_{8_{1}})$ in the 3-manifold  $M_{8_1}(0)$}
\label{equivalent-knots}
\end{figure}
%
%
%
%
\section{Tabulation of annulus presentations of knots}\label{sec:table}
In Section~\ref{sec:main-thm2}, we proved that 
some knots up to 8-crossings have special annulus presentations (see Lemma~\ref{lem:annulus-pre}).
In this section, we give four obstructions for knots to have (special) annulus presentations.
As an application, we prove the following.
\begin{thm}\label{thm:table}
The following knots do not have special annulus presentations:
\[ 5_{1},\  7_{1},\ 7_{3},\ 7_{4},\ 7_{5},\ 8_{2},\
 8_{5},\  8_{8},\  8_{15},\  8_{16},\ 8_{18},\ 8_{19}.\] \end{thm}
\par 
As a summary, we obtain Table~\ref{table1}. 
Here 
\begin{itemize}
\item ``$u=1$" means an unknotting number one knot and it has a special annulus presentation by Lemma~\ref{lem:unknotting}.
\item ``Yes" means a knot of unknotting number $2$ and with a special annulus presentation (see Figure~\ref{annulus-pre}).
\item ``No" means a knot without any special annulus presentations by Theorem~\ref{thm:table}.
\end{itemize}
\begin{table}[h]
  \begin{tabular}{|c|c||c|c|}
  \hline
    knot & special annulus presentation&knot & special annulus presentation  \\  \hline  \hline
    $3_1$ 		& $u=1$	&$8_5$ 		& No 	\\\hline
    $4_1$ 		& $u=1$	&$8_6$ 		&  Yes\\\hline
    $5_1$ 		& No 	&$8_7$ 		& $u=1$	\\\hline
    $5_2$ 		& $u=1$	&$8_8$ 		& No 	\\\hline
    $6_1$ 		& $u=1$	&$8_9$ 		& $u=1$	\\\hline
    $6_2$ 		& $u=1$	&$8_{10}$ 	&  Yes\\\hline
    $6_3$ 		& $u=1$	&$8_{11}$ 	& $u=1$	\\\hline
    $7_1$ 		& No 	&$8_{12}$ 	& Yes	\\\hline
    $7_2$ 		& $u=1$ &$8_{13}$ 	& $u=1$	\\\hline
    $7_3$ 		& No 	&$8_{14}$	& $u=1$	\\\hline
    $7_4$ 		& No 	&$8_{15}$	& No  	\\\hline
    $7_5$ 		& No 	&$8_{16}$ 	& No  	\\\hline
    $7_6$ 		& $u=1$	&$8_{17}$ 	& $u=1$	\\\hline
    $7_7$ 		& $u=1$	&$8_{18}$ 	& No  	\\\hline
    $8_1$ 		& $u=1$	&$8_{19}$ 	& No  	\\\hline
    $8_2$ 		& No 	&$8_{20}$ 	& $u=1$	\\\hline
    $8_3$ 		& Yes &$8_{21}$ 	& $u=1$	\\\hline
    $8_4$ 		& Yes	&-&-\\\hline
\end{tabular}
  \caption{List of prime knots with/without special annulus presentations}
  \label{table1}
\end{table}
In Table $1$, we only consider whether a prime knot up to 8-crossings has a special annulus presentation or not.
In general, a given knot has  many special annulus presentations.
In Section~\ref{sub:equivalent},  we consider when we should regard two special annulus presentations as the same
annulus presentation.
We introduce the  notion of equivalent annulus presentations.


\subsection{The 4-ball genus obstruction}\label{sub1}
The following theorem implies that  the 4-ball genus is 
an obstruction for knots to have annulus presentations.
\begin{thm}\label{thm:4-ball}
Let $K$ be a knot.
If $K$ has an annulus presentation, 
then the $4$-ball genus of $K$, denoted by $g_{4}(K)$, is less than or  equal to one.
\end{thm}

\begin{proof}
Suppose that  $(A,b)$ is an  annulus presentation   of $K$.
Then, a single band surgery along the cocore of  the band $b$
changes  $K$ into $\partial A$. 
Hence, $K$ bounds an orientable proper surface of genus $1$ in $\mathbf{B}^4$.
This means that $g_{4}(K) \le 1$.
\end{proof}

\begin{cor}\label{cor:4-ball}
The following  knots do not have any annulus presentations: 
\[5_1,\ 7_1,\ 7_3,\ 7_5,\ 8_2,\ 8_5,\ 8_{15},\ 8_{19}.\]
\end{cor}
\begin{proof}
Let $K$ be one of the above  knots.
Then $g_4(K)>1$, see KnotInfo \cite{KnotInfo}. By Theorem~\ref{thm:4-ball},
$K$ does not have  any annulus presentations.
\end{proof}
In general, it is a subtle question 
whether a knot $K$ with $g_{4}(K) \le 1$ has (special) annulus presentations or not.
%
\subsection{The concordance obstruction}\label{sub2}
Let $\nu$ be an integer-valued concordance invariant of oriented links satisfying
\begin{align*}
|\nu(L)-\nu(L')|\leq -\chi(S)\  \text{and} \ \nu(H_{\pm})=\pm 1, 
\end{align*}
where $S$ is a concordance between two links $L$ and $L'$, and $H_{+}$ (resp.~$H_{-}$) is the positive 
(resp.~negative) Hopf link. 
For example, the Rasmussen invariant $s$ satisfies this condition,
and we will use this invariant in the proof of Corollary 7.7.
The following theorem implies that $\nu$  is
an obstruction for knots to have (negative)  special annulus presentations.
\begin{thm}\label{thm:concordance}
Let $K$ be an oriented knot.
If $K$ has a negative special annulus presentation, then 
\[ -2 \le \nu(K) \le 0.\]
In particular, if $\nu(K)=2$, then $K$ does not have any negative special annulus presentations.
\end{thm}
\begin{proof}
Suppose that  $K$ has a negative special annulus presentation.
Then, by the definition, $K$ is obtained from the negative Hopf link by a single band surgery. 
Hence, we have 
\[
|\nu(K)-\nu(H_{-})|=|\nu(K)+1|\leq 1. 
\]
That is, $-2 \le \nu(K) \le 0$.
\end{proof}
%
\subsection{The Jones polynomial obstruction}\label{sub3}
The \textit{Jones polynomial} $V_{L}(t)$  is a Laurent polynomial invariant of an oriented link 
$L \subset \mathbf{S}^3$ which is characterized by 
\begin{align} 
V_{U}(t)&=1,\\
t^{-1}V_{L_{+}}( t)-tV_{L_{-}}(t)&=(t^{1/2}-t^{-1/2})V_{L_{0}}(t),\label{eq:jones}
\end{align}
where  the links $L_{+}$, $L_{-}$, $L_{0}$ are identical except for a neighborhood of a point
as shown in Figure~\ref{skein}, for example,  see \cite{Kawauchi}.
\begin{figure}[h!]
\begin{center}
\includegraphics[scale=0.4]{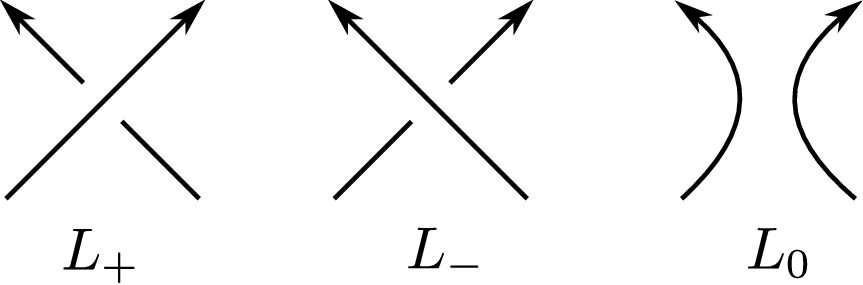}
\end{center}
\caption{A skein triple}
\label{skein}
\end{figure}
{The equation~(\ref{eq:jones})} is called Jones's \textit{skein relation}.
It is well known that the value of $V_{L}(t)$ at a root of unity is related to
topological properties of $L$.
Let $\omega=e^{\frac{\pi i}{3}}$ be the 6th root of unity (not the cube root of unity).
Lickorish and Millett \cite{LM}  described the values of the Jones polynomial  $V_{L}(t)$ at $\omega$
as follows:
\begin{thm}\label{thm:LM}
Let $L$ be an oriented link in $\mathbf{S}^3$, $c$ the number of components of $L$, and $d$ the dimension of 
$H_{1}(\Sigma(L), \mathbf{Z}/\mathbf{3Z})$, where $\Sigma(L)$ is the double  cover of $\mathbf{S}^3$ branched along  $L$.
Then
\[ V_{L}(\omega) = \pm i^{c-1} (i\sqrt{3})^d. \]
\end{thm}
Theorems~\ref{thm1} and \ref{thm2} imply that  the Jones polynomial   is
an obstruction for knots to have (positive) special annulus presentations.
\begin{thm}\label{thm1}
Let $K$ be a knot with $V_{K}(\omega)= - i\sqrt{3}$.
Then  $K$ does not have any positive {special} annulus presentations.
\end{thm}
\par
Note that Theorem~\ref{thm1} is a special case of a more general result in \cite[Theorem~2.2]{Kanenobu}. Here we give a direct proof of Theorem~\ref{thm1} for the sake of the reader.
\begin{proof}
Suppose that  $K$ has a positive {special} annulus presentation.
Then we can suppose that the positive Hopf link $H_{+}$ is obtained from $K$ by a single band surgery as in Figure~\ref{skein3}.
\begin{figure}[h!]
\begin{center} \includegraphics[scale=0.4]{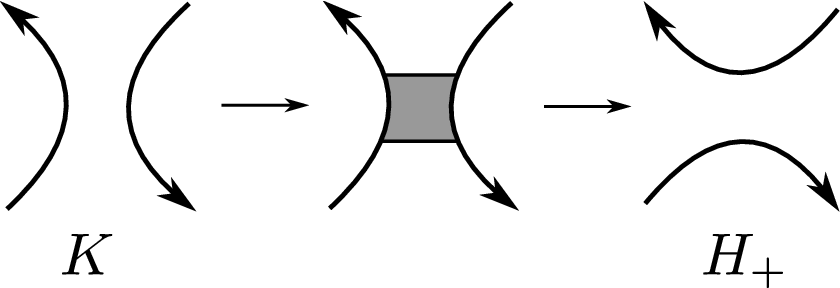} \end{center}
\caption{A band surgery} \label{skein3}
\end{figure}
Here we consider the   skein triple in  Figure~\ref{fig2}.
\begin{figure}[h!]
\begin{center} \includegraphics[scale=0.5]{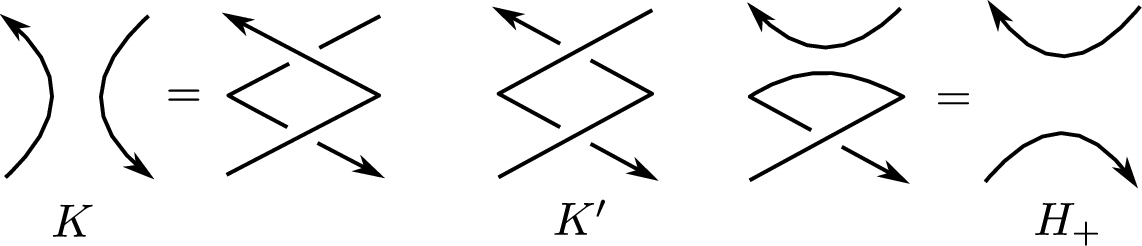} \end{center} \caption{A skein triple} \label{fig2}
\end{figure}
\par
Jones's skein relation implies that 
\[ t^{-1} V_{K}(t) -t  V_{K'}(t)= (t^{1/2}-t^{-1/2})  V_{H_{+}}(t). \]
Since $V_{H_{+}}(t)=-t^{5/2}-t^{1/2}$, this means that $V_{K'}(t)= t^{-2} V_{K}(t)-t^{-1}+1-t+t^2$. 
Hence, we have 
\begin{align*}
V_{K'}(\omega)
&= \omega^{-2} V_{K}(\omega)-\omega^{-1}+1-\omega+\omega^2\\
&=\left(-\frac{1}{2}-\frac{\sqrt{3}}{2}i\right) V_{K}(\omega)-\frac{1}{2}+\frac{\sqrt{3}}{2}i\\
&=  -2 + \sqrt{3}i,
\end{align*}
where we used, for  the last equality,  the hypothesis that $V_{K}(\omega)= - i\sqrt{3}$.
This contradicts Theorem~\ref{thm:LM}. 
Therefore $K$ does not have any positive {special} annulus presentations.
\end{proof}
\begin{cor}\label{cor:Jones}
The knot $7_4$ does not have  any  special annulus presentations.
\end{cor}
\begin{proof}
Let $K$ be the knot $7_4$ in Figure~\ref{knots}.
Then  we can calculate that  $s(K)=2$.
By Theorem~\ref{thm:concordance},
$K$ does not have any negative special annulus presentations.
On the other hand, we have 
\begin{align*}
V_{K}(t)&=t-2t^2+ 3t^3-2t^4+ 3t^5-2t^6+ t^7-t^8. \end{align*}
Therefore  we obtain $V_{K}(\omega)=-i\sqrt{3}$.
By Theorem~\ref{thm1}, 
$K$ does not have any positive special annulus presentations.
\end{proof}
\begin{figure}[h!]
\begin{center}
\includegraphics[scale=0.4]{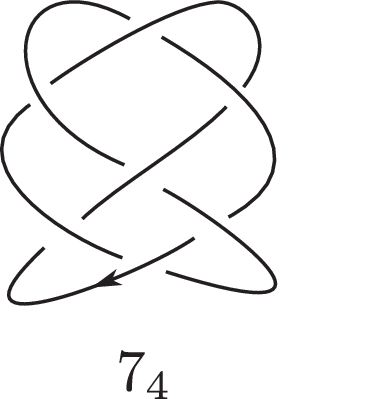}
\end{center}
\caption{The knot $7_4$}
\label{knots}
\end{figure}

\begin{thm}\label{thm2}
Let $K$ be a knot with  $V_{K}(\omega)=\pm3$.
Then $K$ does not have  any special annulus presentations. 
\end{thm}
\begin{proof}
Assume that $K$ has a positive special annulus presentation. 
Then, by the proof of Theorem~\ref{thm1}, we have 
\[ V_{K'}(\omega)
=\left(-\frac{1}{2}-\frac{\sqrt{3}}{2}i\right) V_{K}(\omega)+\frac{\sqrt{3}}{2}i-\frac{1}{2}.\] 
Both cases $V_{K}(\omega)=\pm3$ contradict  Theorem~\ref{thm:LM}.
Hence, $K$ does not have  any positive special annulus presentations. 
Let $\overline{K}$ be the mirror image of $K$. 
Since $V_{\overline{K}}(\omega)=\overline{V_{K}(\omega)}=\pm 3$, 
we see that $\overline{K}$ also does not have  any positive special annulus presentations. 
Equivalently, $K$ does not have any negative special annulus presentations. 
\end{proof}
\begin{cor} \label{cor:Jones2}
The  knot $8_{18}$ does not have any special annulus presentations. 
\end{cor}
\begin{proof}
We have $V_{8_{18}}(\omega)=\pm3$.
By Theorem~\ref{thm2}, 
the knot $8_{18}$ does not have any negative special annulus presentations. 
\end{proof}
\subsection{The $Q$-polynomial obstruction}\label{sub4}
The \textit{$Q$-polynomial} $Q_{L}(x)$
is a Laurent polynomial invariant of an unoriented link $L \subset \mathbf{S}^3$
which is characterized by 
 \begin{align*}
Q_{U}(x)&=1,\\
Q_{L_{+}}(x)+Q_{L_{-}}(x)&=x(Q_{L_{0}}(x)+Q_{L_{\infty}}(x)),
\end{align*}
where  the links $L_{+}$, $L_{-}$, $L_{0}$, $L_{\infty}$ are identical except for a neighborhood of a point
as shown in Figure~\ref{skein7}, for example, see \cite{Kawauchi}.
\begin{rem} 
The $Q$-polynomial $Q_{L}(x)$ have the following properties:
\begin{itemize}
\item $Q_{L}(x)=F_{L}(1,x)$, where $F_{L}(a,z)$ is the Kauffman polynomial of $L$, 
\item $Q_{L}(x)=Q_{\overline{L}}(x)$, where $\overline{L}$ is the mirror image of $L$. 
\end{itemize}
\end{rem}
\begin{figure}[h!]
\begin{center}
\includegraphics[scale=0.4]{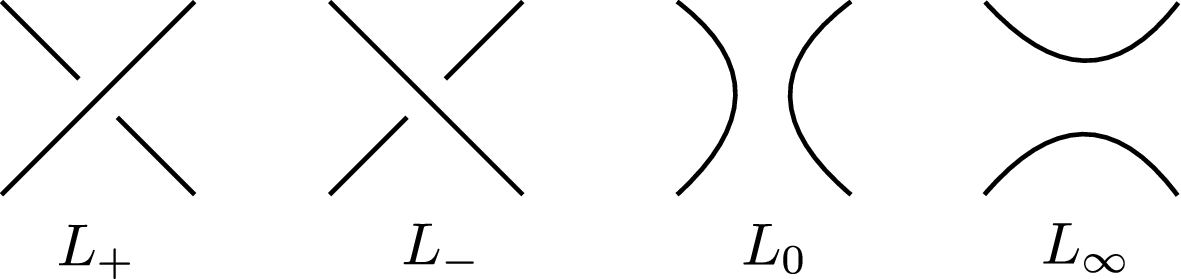}
\end{center}
\caption{A skein quadruple}
\label{skein7}
\end{figure}
An analogous result  of Theorem~\ref{thm1} holds for the Q-polynomial.

\begin{thm}\label{thm:Q}
Let $K$ be a knot with a special annulus presentation.
Then {we have} $Q_{L}\left(\dfrac{\sqrt{5}-1}{2}\right)\neq \sqrt{5}$.
\end{thm}

\begin{proof}
Based on the result of Long in \cite{Rong} (see also Stoimenow \cite{Stoimenow}),
Kanenobu \cite[Theorem~3.1]{Kanenobu} proved that 
if two links $L_1$ and $L_2$ are related by a single band surgery,
then 
\[ Q_{L_1}\left(\dfrac{\sqrt{5}-1}{2}\right)\Big/ Q_{L_2}\left(\dfrac{\sqrt{5}-1}{2}\right)
\in \{ \pm1, \sqrt{5}^{\pm1} \}.\]
Since the knot $K$ and the Hopf link $H$  are  related by a single band surgery, we obtain
\[ Q_{K}\left(\dfrac{\sqrt{5}-1}{2}\right)\Big/ Q_{H}\left(\dfrac{\sqrt{5}-1}{2}\right)
\in \{  \pm1, \sqrt{5}^{\pm1}\}.\]
Since $Q_{H}\left(\dfrac{\sqrt{5}-1}{2}\right)=-1$, this implies   that
$ Q_{K}\left(\dfrac{\sqrt{5}-1}{2}\right) \in \{ \pm1, -\sqrt{5}^{\pm1} \}$.
In particular, $Q_{K}\left(\dfrac{\sqrt{5}-1}{2}\right)\neq \sqrt{5}$.
\end{proof}

\begin{cor} \label{cor:table-final}
The knots $8_{8}$ and  $8_{16}$ do not have any special annulus presentations.
\end{cor}

\begin{proof}
 We have the following.
\begin{align*}
%
Q_{8_8}(x)&=1+4x+6x^{2}-10x^{3}-14x^{4}+4x^{5}+8x^{6}+2x^{7}, \\
Q_{8_{16}}(x)&=-3+10x+18x^{2}-22x^{3}-30x^{4}+8x^{5}+16x^{6}+4x^{7}.
\end{align*}
Therefore 
\begin{align*}
Q_{8_8}\left(\dfrac{\sqrt{5}-1}{2}\right)= 
Q_{8_{16}}\left(\dfrac{\sqrt{5}-1}{2}\right)= 
\sqrt{5}.
\end{align*}
By Theorem~\ref{thm:Q},
the knots $8_{8}$ and $8_{16}$ do not have any  special annulus presentations.
\end{proof}

\begin{proof}[Proof of Theorem~\ref{thm:table}]
By Corollaries \ref{cor:4-ball}, \ref{cor:Jones}, \ref{cor:Jones2} and   \ref{cor:table-final}, 
the knots 
\[ 5_{1},\  7_{1},\ 7_{3},\ 7_{4},\ 7_{5},\ 8_{2},\
 8_{5},\  8_{8},\  8_{15},\  8_{16},\ 8_{18},\ 8_{19}\]
 do not have special annulus presentations.
\end{proof}

\begin{rem} 
There exist knots which have only non-special annulus presentations. 
By Theorem~7.1, the knots $7_4$, $8_8$, $8_{16}$ and $8_{18}$ do not have any special annulus presentations. 
On the other hand, it is easy to see that $7_4$ has a non-special annulus presentation (see Figure~\ref{fig:non-special}). 
Also, we can obtain a non-special annulus presentation of $8_8$ from a ribbon presentation of $8_8$. 
In fact, we can find a ribbon presentation of $8_8$ in \cite[Appendix~F.5]{Kawauchi}, and this presentation is also a non-special annulus presentation of $8_8$. 
The authors do not know whether the knots $8_{16}$ and $8_{18}$ have non-special annulus presentations or not. 
\end{rem}

\begin{figure}[h]
\begin{center}
\includegraphics[scale=0.3]{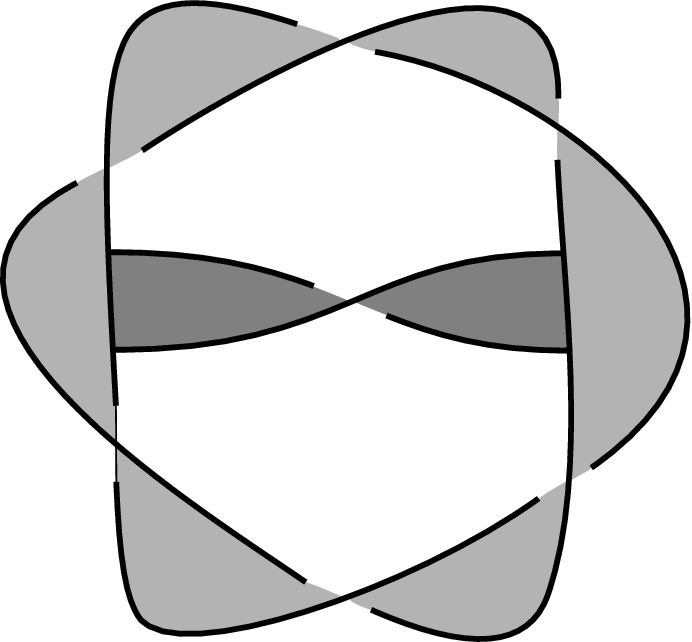}
\caption{A non-special annulus presentation of $7_4$}
\label{fig:non-special}
\end{center}
\end{figure}

\subsection{Equivalent annulus presentations of knots}\label{sub:equivalent}
In general, a given knot has  many special annulus presentations.
For example, the knot $6_3$ has two special annulus presentations $(A,b_1)$ and $(A,b_2)$
as in Figure~\ref{figure:annulus-pre2}.
\begin{figure}[h!]
\centering
\includegraphics[scale=0.7]{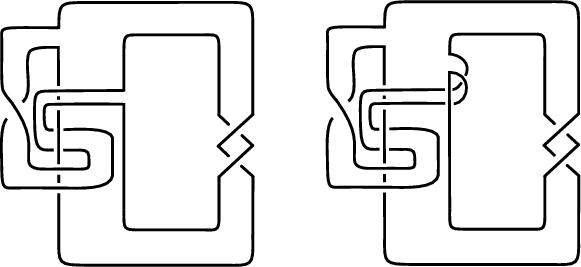}
\caption{Two special annulus presentations of $6_3$}
\label{figure:annulus-pre2}
\end{figure}

We define as follows.
\begin{defn}\label{def:equivalent}
Let $(A_i,b_i)$ be  an annulus presentation of a knot $K_i$ for $i=1, 2$. 
Then $(A_1,b_1)$ and $(A_2,b_2)$ are {\it equivalent} if 
the $3$-component link $K_1 \cup c_1'\cup c_2'$ is isotopic to $K_2\cup d_1'\cup d_2'$,
where $K_1$ is deformed into  $K_2$, and  $A_1'$, $A_2'$ are the  shrunken annuli corresponding to  $(A_1,b_1)$, $(A_2,b_2)$  respectively 
and  $\partial A_1'=c_1'\cup c_2'$, $\partial A_2'=d_1'\cup d_2'$. 
\end{defn}

The following theorem justifies the above definition.
\begin{thm}\label{thm:equivalent}
Let  $(A_1,b_1)$ and $(A_2,b_2)$ be equivalent annulus presentations of $K_1$ and $K_2$, respectively. 
Then  we have 
\[A_1^{n}(K_1)=A_2^{n}(K_2), \text{or}\  A_1^{n}(K_1)=A_2^{-n}(K_2) \]
 for any $n\in \mathbf{Z}$. 
\end{thm}

We omit the proof of Theorem~\ref{thm:equivalent}
since it is immediately  follows from Definition~\ref{def:equivalent}.
The following lemma is useful to find  equivalent annulus presentations.

\begin{lem}\label{lem:equivalent}
Let $A$ be (possibly)  twisted and knotted annulus in $\mathbf{S}^3$.
Let $(A,b_1)$ and $(A,b_2)$ be two annulus presentations of a knot 
whose bands $b_1$ and $b_2$  are slightly different as  in Figure~\ref{figure:equivalent}.  
Then  $(A,b_1)$ and $(A,b_2)$ are equivalent.
\end{lem}
\begin{figure}[h]
\centering
\includegraphics[scale=0.5]{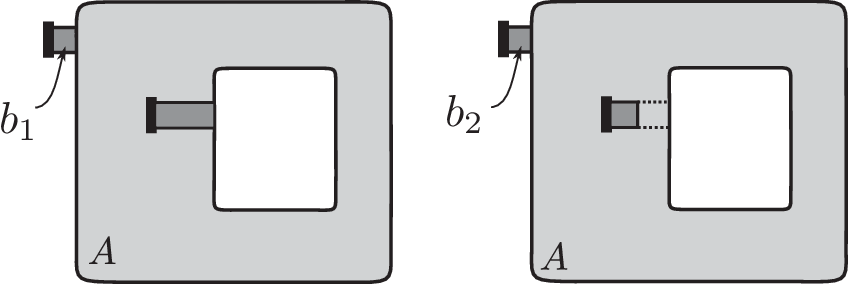}
\caption{$A$ may be knotted and twisted. }
\label{figure:equivalent}
\end{figure}
\begin{proof}
See Figure~\ref{figure:equivalent2}.
\begin{figure}[h]
\centering
\includegraphics[scale=0.55]{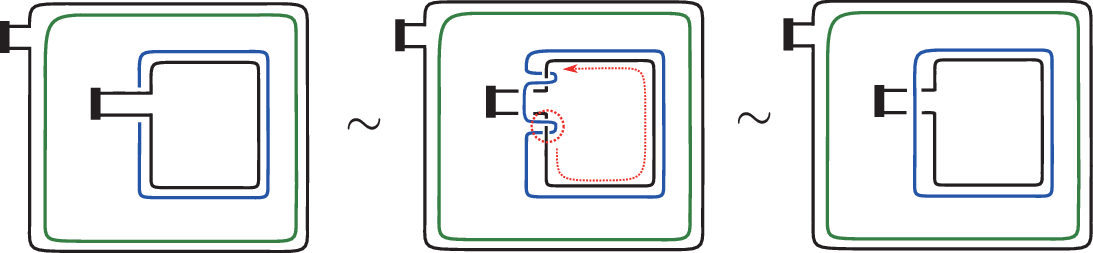}
\caption{(color online) Proof of Lemma~\ref{lem:equivalent} }
\label{figure:equivalent2}
\end{figure}
\end{proof}
\begin{ex}
Two special annulus presentations $(A,b_1)$ and $(A,b_2)$ in Figure~\ref{figure:annulus-pre2} are equivalent.
\end{ex}
We do not know whether  the converse of Theorem~\ref{thm:equivalent} holds or not. 
\begin{question}
Let  $(A_1,b_1)$ and $(A_2,b_2)$ be annulus presentations of $K_1$ and $K_2$, respectively. 
If $A^{n}(K_1)=A^{n}(K_2)$ for any $n\in \mathbf{Z}$, then are $(A_1,b_1)$ and $(A_2,b_2)$ equivalent?
\end{question}
%


\noindent
\textbf{Acknowledgements.}
The final part of this paper was written in OIST.
The first  author thanks Andrew Lobb for inviting him  to 
``Mini-Symposium : Knot Theory on Okinawa'' during 17--21 February 2020.
He also thanks Kazuhiro Ichihara for telling him the paper \cite{CM} and
Chuck Livingston for clarifying some confusing points on orientations in Section \ref{sec:annulus}.
The authors thank the referee for his/her careful reading and helpful comments.
The first  author was supported by the Research Promotion Program for Acquiring Grants in-Aid for Scientific Research (KAKENHI)
in Ritsumeikan University.
The second author was supported by JSPS KAKENHI Grant number JP18K13416. 
\newpage
{\section*{Appendix}}
{We give
a complete proof of Theorem 3.1.}

{\begin{proof}[Proof of Theorem 3.1]
A desired homeomorphism $\psi_{m}\colon M_{K}(m) \to M_{T_{m}(A(K))}(m)$ is given as in Figure~\ref{fig:negative} if $A$ is the negative Hopf band and as in Figure~\ref{fig:positive3} if $A$ is the positive Hopf band.  
\end{proof}
}

\begin{figure}[h!]
\begin{center}
\includegraphics[scale=0.88]{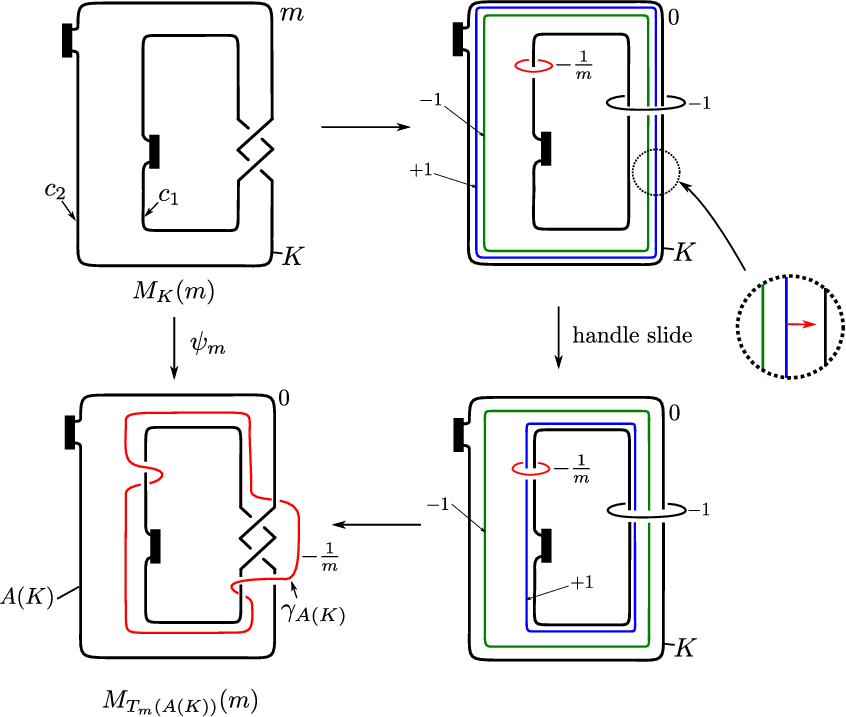}
\caption{(color online) $\psi_{m}\colon M_{K}(m) \to M_{T_{m}(A(K))}(m)$ for the case where A is the negative Hopf
band}
\label{fig:negative}
\end{center}
\end{figure}

{\begin{rem}
Note that $\psi_{m}$ in Figures~\ref{fig:negative} and \ref{fig:positive3} do not depend on the choices of the meridians with slopes $-1/m$ of $K$. 
For example, Figures~\ref{fig:positive3} and \ref{fig:positive2} are essentially the same. 
\end{rem}}

\newpage

\begin{figure}[h!]
\begin{center}
\includegraphics[scale=0.87]{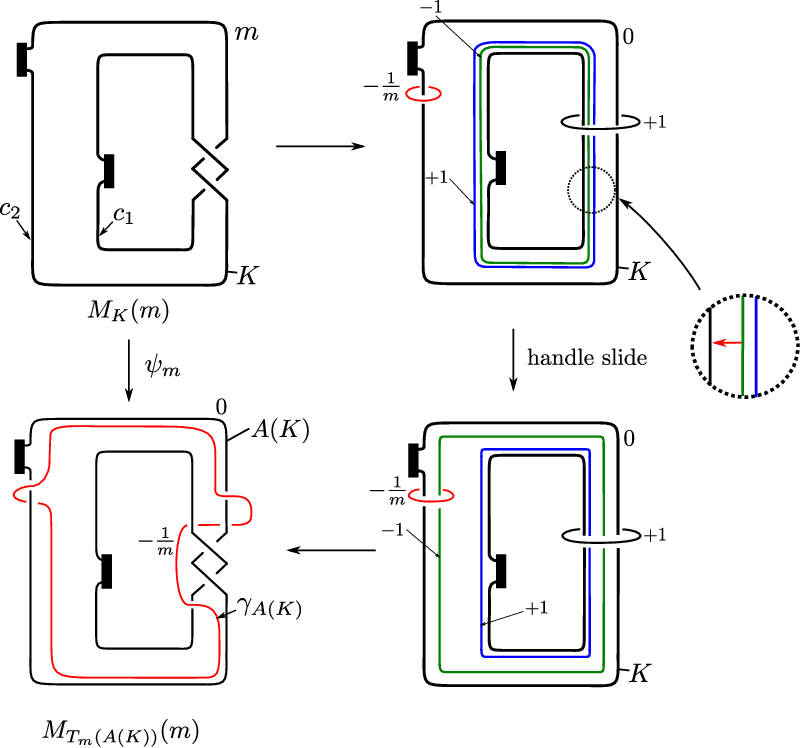}
\caption{(color online) $\psi_{m}$ for the case where $A$ is the positive Hopf
band}
\label{fig:positive3}
\end{center}
\end{figure}

\begin{figure}[h!]
\begin{center}
\includegraphics[scale=0.77]{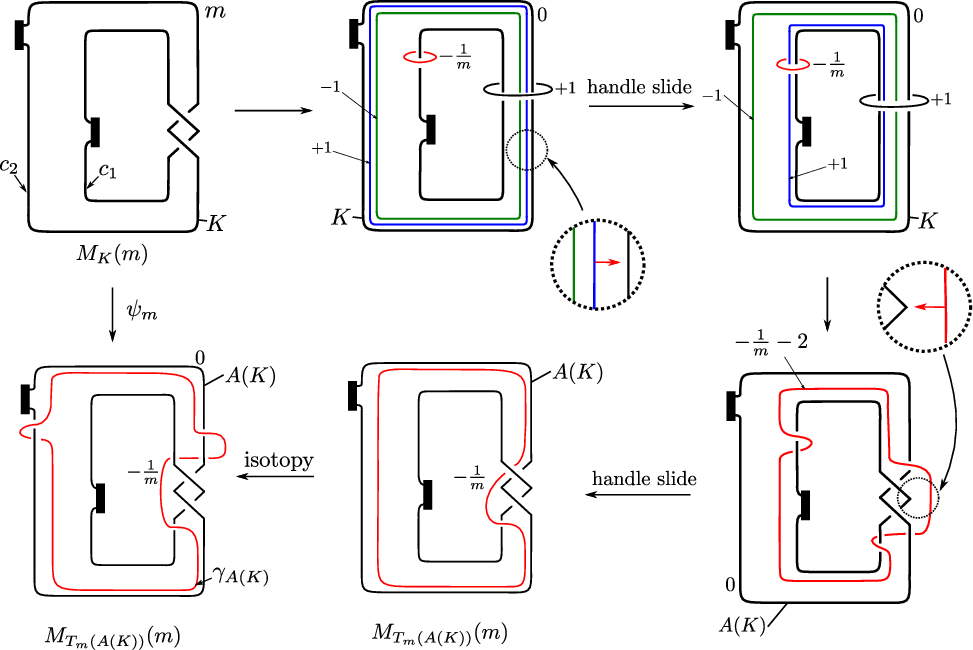}
\caption{(color online)}
\label{fig:positive2}
\end{center}
\end{figure}

\end{document}